\documentclass[12pt]{amsart}   	
\usepackage{geometry}                		
\usepackage{graphicx}				
	\usepackage{color}							
\usepackage{mathrsfs}
\usepackage{graphicx}
\usepackage{amsmath, amsthm, amssymb}
\usepackage{amsfonts}
\usepackage{hyperref} 
\def\blue{\textcolor{blue}}
\def\red{\textcolor{red}}
\newtheorem{thm}{Theorem}
\newtheorem{cor}[thm]{Corollary}

\newtheorem{prop}[thm]{Proposition}
\newtheorem{conj}[thm]{Conjecture}
\theoremstyle{definition}
\newtheorem{defn}[thm]{Definition}
\newtheorem*{rmk}{Remark}

\numberwithin{equation}{section}

\def\Z{\mathbb{Z}}
\def\Q{\mathbb{Q}}
\def\N{\mathbb{N}}

\title[A curious interpolation of  Carlitz-Riordan's  $q$-ballot numbers]{A curious polynomial interpolation of \\ Carlitz-Riordan's  $q$-ballot numbers}
 \author{Fr\'ed\'eric Chapoton}
 \author{Jiang Zeng}
 \address{Universit\'{e} de Lyon,  Universit\'{e} Lyon 1,  Institute Camille Jordan,
UMR 5028 du CNRS,  69622 Villeurbanne, France}
\email{chapoton@math.univ-lyon1.fr, zeng@math.univ-lyon1.fr}

\date{\today}							

\begin{document}
\maketitle
\begin{abstract}
  We study a polynomial sequence $C_n(x|q)$ defined as a solution of a
  $q$-difference equation. This sequence, evaluated at $q$-integers,
  interpolates Carlitz-Riordan's $q$-ballot numbers. In the basis given by
  some kind of $q$-binomial coefficients, the coefficients are again
  some $q$-ballot numbers. We obtain in a combinatorial way another
  curious recurrence relation for these polynomials. 
\end{abstract}

\section{Introduction}

This paper was motivated by  a previous work
of the first author on flows on rooted trees \cite{Cha13b}, where the well-known
Catalan numbers and the closely related ballot numbers played an
important role. In fact, one can easily introduce one more parameter
$q$ in this work, and then Catalan numbers and ballots numbers get
replaced by their $q$-analogues introduced a long time ago by Carlitz-Riordan~\cite{CR64}, see also \cite{Car72,FH85}.

These $q$-Catalan numbers have been recently considered by many people,
see for example \cite{Cig05, BF07, Hag08, BP12}, including some work by Reineke \cite{Rein05} on
moduli space of quiver representations.

Inspired by an analogy with another work of the first author on rooted
trees \cite{Cha13a}, it is natural to try to interpolate the $q$-ballot
numbers. In the present article, we prove that this  is possible and
study the interpolating polynomials.

Our main object of study is a sequence of polynomials in $x$ with
coefficients in $\Q(q)$, defined by the  $q$-difference
equation:
\begin{align}\label{eq1}
  \Delta_{q}C_{n+1}(x|q)=qC_n(q^2x+q+1|q),
\end{align}
where $\Delta_qf(x)={(f(1+qx)-f(x))}/{(1+(q-1)x)}$ is the Hahn operator.

After reading a previous version of this paper, Johann Cigler has kindly  brought the two related references \cite{Cig97, Cig98} to our attention, where 
a sequence of more general polynomials $G_n(x,r)$ was introduced  through 
a $q$-difference operator for $q$-integer $x$ and positive integer $r$. Comparing these two sequences  one has 
$$G_n(qx+1,2)=C_{n+1}(x|q).$$ 

In  the  next  section  we  recall classical  material  on  Carlitz-Riordan's
$q$-analogue  for   Catalan  and   ballot  numbers  and   define  our
polynomials.  In the third  section, we  evaluate our  polynomials at
$q$-integers in  terms of $q$-ballot numbers and  prove a product formula when
 $q=1$. In the fourth section, we find their expansion  in a
basis made of  a kind of $q$-binomial coefficients  and obtain another
recurrence for these polynomials. This recurrence  is 
 not usual   even in the special $x=q=1$ case and  we have only a combinatorial proof in the general case.
We conclude the paper with some open problems.

\textsc{Nota Bene:} Figures are best viewed in color.

\section{Carlitz-Riordan's $q$-ballot numbers}
Recall that the Catalan numbers $C_n=\frac{1}{n+1}{2n\choose n}$ may be defined as solutions to  
\begin{align}\label{def-cat}
C_{n+1}=\sum^{n}_{k=0}C_kC_{n-k}, \quad (n\geq 0),\quad C_0=1.
\end{align} 
 The first values are
$$
\begin{array}{c|cccccccccc}
n&0&1&2&3&4&5&6&7&8\\
\hline
C_n&1&1&2&5&14&42&132&429&1430
\end{array}
$$
It is well known that $C_n$ is the number of lattice paths 
from $(0,0)$ to $(n,n)$  with steps $(1,0)$ and $(0,1)$, which do not pass above
the line $y=x$.   As a natural generalization, one considers
the set ${\mathcal P}(n,k)$ of    lattice paths  from $(0,0)$ to $(n+1,k)$ 
with steps $(1,0)$ and $(0,1)$,  such that  the last step  is $(1,0)$
and  they  never rise above the line $y=x$. 
Let $f(n,k)$ be the cardinality of 
 ${\mathcal P}(n,k)$.  The first values of $f(n,k)$ are given in Table~\ref{table1}. These numbers are called  
 ballot numbers and have a long history in the literature of combinatorial theory. 
 Moreover,  one (see \cite{Com74}) has the explicit formula
\begin{align}\label{eq:ballot}
f(n,k)=\frac{n-k+1}{n+1}{n+k\choose k}\qquad (n\geq k\geq 0).
\end{align}

\begin{table}
$$
\begin{array}{c|ccccccc}
$n\textbackslash k$&0&1&2&3&4&5&6\\
\hline
0&{\bf1}&&&&&&\\
1&1&{\bf1}&&&&&\\
2&{\bf1}&2&{\bf 2}&&&&\\
3&1&{\bf 3}&5&{\bf 5}&&&\\
4&{\bf1}&4&{\bf 9}&14&{\bf 14}&&\\
5&1&{\bf 5}&14&{\bf 28}&42&{\bf 42}&\\
6&{\bf1}&6&{\bf 20}&48&{\bf 90}&132&{\bf 132}\\
\end{array}
$$
\caption{\label{table1} The first values of ballot numbers $f(n,k)$}
\end{table}

Carlitz and Riordan~\cite{CR64} introduced the following  $q$-analogue of these numbers
\begin{align}\label{def:carlitz}
 f(n,k|q)=\sum_{\gamma\in {\mathcal P}(n,k)}q^{A(\gamma)},
\end{align}
 where $A(\gamma)$ is the area under the path (and above the $x$-axis).
The first values of $f(n,k|q)$ are given in Table~\ref{table2}.
\begin{table}
$$
\begin{array}{c|ccccc}
$n\textbackslash k$&0&1&2&3&4\\
\hline
0&1&&&&\\
1&1&q&&&\\
2&1&q+q^2&q^2+q^3&&\\
3&1&q+q^2+q^3&q^2+q^3+2q^4+q^5&q^3+q^4+2q^5+q^6\\
4&1&q+q^2+q^3+q^4&q^2+q^3+2q^4+2q^5+2q^6+q^7&q^3Y&q^4Y
\end{array}
$$
\caption{\label{table2} The first values of $q$-ballot numbers $f(n,k|q)$ with $Y=q^6+3q^5+3q^4+3q^3+2q^2+q+1$}
\end{table}
Furthermore, Carlitz \cite{Car72} uses a variety of elegant techniques to derive several basic properties of the $f(n,k|q)$, among 
which the following is the basic recurrence relation
\begin{align}\label{eq:fq}
f(n,k|q)=qf(n,k-1|q)+q^{k}f(n-1,k|q)\qquad (n, k\geq 0),
\end{align}
where $f(n,k|q)=0$ if $n<k$ and $f(0,0|q)=1$.

It is also easy to see that   the polynomial $f(n,k|q)$ is of degree $kn-k(k-1)/2$ and  satisfies the equation
$f(n,n|q)=qf(n,n-1|q)$.
If one defines the $q$-Catalan numbers by 
\begin{align}
C_{n+1}(q)= \sum_{k=0}^nf(n,k|q)=q^{-n-1}f(n+1,n+1|q)\qquad (n\geq 0),
\end{align}
then, one obtains the following analogue of \eqref{def-cat} for the Catalan numbers
\begin{align}\label{eq:carlitz}
C_{n+1}(q)=\sum_{i=0}^n  C_{i}(q)C_{n-i}(q)q^{(i+1)(n-i)},
\end{align}
where $C_0(q)=1$. Setting  ${\widetilde C}_n(q)=q^{n\choose 2} C_{n}(q^{-1})$, one has a simpler $q$-analog of
\eqref{def-cat} 
\begin{align}
 {\widetilde C}_{n+1}(q)=\sum_{i=0}^n q^{i}{\widetilde C}_{i}(q) {\widetilde C}_{n-i}(q).
\end{align}

The first values are $C_1(q)=1$, $C_2(q)=1+q$, $C_3(q)=1+q+2q^2+q^3$ and 
$$
C_4(q)=1+q+2q^2+3q^3+3q^4+3q^5+q^6.
$$
No explicit formula is known for  Carlitz-Riordan's $q$-Catalan numbers. However,
Andrews~\cite{And75} proved the following recurrence formula
\begin{align}
C_n(q)=\frac{q^n}{[n+1]_q}{2n\brack n}_q+q\sum_{j=0}^{n-1}(1-q^{n-j}) q^{(n+1-j)j}{2j+1\brack j}_q C_{n-1-j}(q),
\end{align}
where $[x]_{q}=\frac{q^{x}-1}{q-1}$.

Recall that the $q$-shifted factorial $(x;q)_n$ is defined by
$$
(x;q)_n=(1-x)(1-xq)\cdots (1-xq^{n-1})\quad  (n\geq 1)\quad \text{and}\quad  (x;q)_0=1.
$$
The two  kinds of $q$-binomial coefficients are defined  by
\begin{align*}
{n\brack k}_q:=\frac{(q;q)_n}{(q;q)_k(q;q)_{n-k}},\qquad
{x\choose k}_q:=\frac{x(x-1)\ldots (x-[k-1]_q)}{[k]_q!},
\end{align*}
with ${x\choose 0}_q=1$.
Note that 
$$
{[n]_q\choose k}_q=q^{k\choose 2} {n\brack k}_q,\qquad 
{[-n]_{q}\choose k}_q=(-1)^{k}q^{-kn}{k+n-1\brack k}_{q}.
$$

The $q$-derivative operator ${\mathcal D}_q$ and Hahn operator $\Delta_q$ are  defined by
\begin{align}\label{Hahn}
{\mathcal D}_qf(x)=\frac{f(qx)-f(x)}{(q-1)x}\quad\text{and} \quad \Delta_qf(x)=\frac{f(1+qx)-f(x)}{1+(q-1)x}.
\end{align}

\begin{defn}
The  sequence of  polynomials $\{C_n(x|q)\}_{n\geq 1}$ is defined 
by the $q$-difference equation~\eqref{eq1}  or equivalently
\begin{align}\label{eq:def1}
\frac{C_{n+1}(x|q)-C_{n+1}(q^{-1}x-q^{-1}|q)}{1+(q-1)x}=C_n(qx+1|q)\qquad (n\geq 1),
\end{align}
with the initial condition $C_1(x|q)=1$ and  $C_{n}(-\frac{1}{q}|q)=0$ for $n\geq 2$.
\end{defn}
For example, we have 
\begin{align*}
C_2(x|q)&=1+q{x\choose 1}_q,\\
C_{3}(x|q)&=(1+q)+(q+q^2+q^3){x\choose 1}_q+q^{4}{x\choose 2}_q,\\
C_{4}(x|q)&=(q^3+q^2+2q+1)+(q^6+q^5+2q^4+2q^3+2q^2+q){x\choose 1}_q\\
&\hspace{2 cm} +(q^9+q^8+q^7+q^6+q^5)q^{-1}{x\choose 2}_q+q^{9} {x\choose 3}_q.
\end{align*}
It is clear that $C_{n}(x|q)$ is a polynomial in $\Q(q)[x]$ of degree $n-1$ for $n\geq 1$. 


\section{Some preliminary results}
We first show that the evaluation of the polynomials $C_{n}(x|q)$ at $q$-integers is always a polynomial 
in $\N[q]$.
Note that  formulae \eqref{eq:case1} and  \eqref{eq:explicit}  were  implicitly given in \cite{Cig97}.

\begin{prop} \label{prop1} When $x=[k]_{q}$ we have 
\begin{align}\label{eq:case1}
C_{n+1}([k]_{q}|q)=q^{kn+\frac{n(n+1)}{2}}f(k+n, n|q^{-1})\qquad (n,\, k\geq 0).
\end{align}
\end{prop}
\begin{proof} When $x=[k]_{q}$ Eq.~\eqref{eq:def1} becomes
\begin{align}
C_{n+1}([k]_{q}|q)=q^k C_{n}([k+1]_{q}|q)+C_{n+1}([k-1]_{q}|q).
\end{align}
It is easy to see that  \eqref{eq:case1} is equivalent to  \eqref{eq:fq}.
\end{proof}

\begin{cor}
We have 
\begin{align}
C_{n+1}(0|q)=C_{n}(1|q)\quad\text{and} \quad C_{n+1}(1|q)={\widetilde C}_{n+1}(q).
\end{align}
\end{cor}
\begin{proof}
Letting $x=0$ in \eqref{eq:def1} we get $C_{n+1}(0|q)=C_{n}(1|q)$. 
Letting $k=1$ in \eqref{eq:case1} we have 
\begin{align*}
C_{n+1}(1|q)&=q^{n+\frac{n(n+1)}{2}}f(1+n, n|q^{-1})\\
&=q^{n+1+\frac{n(n+1)}{2}} f(n+1, n+1|q^{-1})\\
&=q^{\frac{n(n+1)}{2}} C_{n+1}(q^{-1}),
\end{align*}
which is equal to  ${\widetilde C}_{n+1}(q)$ by definition.
\end{proof}
The shifted factorial is defined by
$$
(x)_0=1\quad \text{and} \quad (x)_n=x(x+1)\cdots (x+n-1), \quad n=1, 2, 3,\ldots,
$$
and $(x)_{-n}=1/(x-n)_n$.

\begin{prop}\label{prop2}
When $q=1$ we have the explicit formula
\begin{align}\label{eq:explicit}
C_{n+1}(x|1)=\frac{(x+1) (x+n+2)_{n-1}}{n!}=\frac{x+1}{x+1+n}{x+2n\choose n} \quad (n\geq 0).
\end{align}
\end{prop}
\begin{proof} 
When $q=1$ the equation~\eqref{eq:def1} reduces to
\begin{align}\label{c:q=1}
C_{n+1}(x|1)=C_{n+1}(x-1|1)+C_{n}(x+1|1).
\end{align}
Since $C_{n+1}(x|1)$ is a polynomial in $x$ of degree $n$, it suffices to prove that 
the right-hand side of \eqref{eq:explicit} satisfy \eqref{c:q=1} for 
$x$ being positive integers $k$. 
By Proposition~\ref{prop1} and \eqref{eq:ballot} it suffices  to 
check the  following identity 
\begin{align}
\frac{k+1}{k+1+n}{k+2n\choose n}=\frac{k}{k+n}{k-1+2n\choose n}+\frac{k+2}{k+1+n}{k+2n-1\choose n-1}.
\end{align}
This is straightforward.
\end{proof}


To motivate our result in the next section we first prove two  $q$-versions of a folklore result on 
the  polynomials which take integral values on integers (see \cite[p. 38]{St86}).
Introduce the polynomials $p_{k}(x)$ by
$$
p_0(x)=1\quad \text{and}\quad p_k(x)=(-1)^kq^{-{k\choose 2}}\frac{(x-1)(x-q)\cdots (x-q^{k-1})}{(q;q)_k},\quad 
k\geq 1.
$$
So $p_k(q^n)={n\brack k}_q$ for  $n\in \N$.

\begin{prop}\label{prop3} The following statements hold true.
\begin{enumerate}
\item[(i)] The polynomial $f(x)$ of degree $k$ assumes values in $\Z[q]$ at $x=1, q,\ldots, q^k$ if and only if 
\begin{align}\label{eq:Newton-basis1}
f(x)=c_0+c_1p_1(x)+\cdots +c_kp_k(x),
\end{align}
where $c_j=q^{{j\choose 2}}(1-q)^j {\mathcal D}^j_qf(1)$ are polynomials in $\Z[q]$  
 for $0\leq j\leq k$.
\item[(ii)] The polynomial ${\tilde f}(x)$ of degree $k$ assumes values in $\Z[q]$ at $x=0, [1]_q,\ldots, [k]_q$ if and only if 
\begin{align}\label{eq:Newton-basis2}
{\tilde f}(x)=\sum_{j=0}^k{\tilde c}_j q^{-{j\choose 2}}{x\choose j}_q,
\end{align}
where  ${\tilde c}_j=q^{{j\choose 2}}\Delta_q^j{\tilde f}(0)$ are polynomials in $\Z[q]$
for $0\leq j\leq k$.
\end{enumerate}
\end{prop}
\begin{proof}
Clearly we can expand any polynomial $f(x)$ of degree $k$ in the basis $\{p_j(x)\}_{0\leq j\leq k}$ as in 
\eqref{eq:Newton-basis1}. 
Besides, it is easy to see that
\begin{align}\label{q-derivative}
{\mathcal D}_qp_k(x)=\frac{q^{1-k}}{1-q}p_{k-1}(x)\Longrightarrow 
{\mathcal D}^j_qp_k(x)=\frac{q^{{j+1\choose 2}-jk}}{(1-q)^j} p_{k-j}(x).
\end{align}
Hence, applying ${\mathcal D}^j_q$ to the two sides of \eqref{eq:Newton-basis1} we obtain
$$
{\mathcal D}^j_qf(1)=c_j \frac{q^{-{j\choose 2}}}{(1-q)j}\Longrightarrow 
c_j=q^{{j\choose 2}}(1-q)^j {\mathcal D}^j_qf(1).
$$
Since ${\mathcal D}^j_qf(1)$ involves only the values of $f(x)$ at $x=0,[1],\ldots, [k]_q$ for $0\leq j\leq k$,
the result follows.  In the same manner, since
$$
\Delta_q{x\choose k}_q={x\choose k-1}_q,
$$
we obtain the expansion  \eqref{eq:Newton-basis2}.
\end{proof}
\begin{rmk} 
\begin{enumerate}
\item We can also derive \eqref{eq:Newton-basis2}  from \eqref{eq:Newton-basis1}
as follows.
 Let $y=\frac{x-1}{q-1}$. For any polynomial $f(x)$ define
$\tilde f(y)=f(1+(q-1)y)$.
Since  $q^n=1+(q-1)[n]_q$, it is clear that $f(q^n)\in \Z[q]\Longleftrightarrow \tilde f([n]_q)\in \Z[q]$.
Writing  $1+qx-[j]_q=q(x-[j-1]_q)$ we see that 
$$
\tilde f_j(x)=q^{-{j\choose 2}} {x\choose j}_q.
$$
The expansion \eqref{eq:Newton-basis2}  follows from \eqref{eq:Newton-basis1} immediately.
\item
When  $\tilde f(x)=x^n$, it is known (see, for example,  \cite{Ze06}) that
$$
\tilde c_k=\Delta_q0^n=[k]_q! S_q(n,k),
$$
where $[k]_q!=[1]_q\cdots [k]_q$ and $S_q(n,k)$ are classical $q$-Stirling numbers of the second kind defined 
by
$$
S_q(n,k)=S_q(n-1,k-1)+[k]_qS_q(n-1,k)\qquad \text{for}\quad n\geq k\geq 1,
$$
with $S_q(n,0)=S_q(0,k)=0$ except $S_q(0,0)=1$.
\item The two formulas \eqref{eq:Newton-basis1} and \eqref{eq:Newton-basis2} are special cases of 
 the Newton interpolation formula, namely,  for any polynomial $f$ of degree less than or equal to $n$ one has
\begin{align}\label{newtoninterpol}
f(x)=\sum_{k=0}^n \Biggl( \sum_{j=0}^k\frac{f(b_j)}
{\prod_{r=0, r\neq j}^{k} (b_j-b_r)} \Biggr) (x-b_0)\cdots (x-b_{k-1}),
\end{align}
where $b_0, b_1, \ldots, b_{n-1}$ are  distinct complex numbers.  Some recent applications of \eqref{newtoninterpol}
in the computation of moments of Askey-Wilson polynomials  are given in \cite{GITZ}.
\end{enumerate}
\end{rmk}
\section{Main results}
In the light of Propositions~\ref{prop1} and \ref{prop3}, it is natural to consider the expansion of
$C_{n+1}(x|q)$ and $C_{n+1}(qx+1|q)$ in the basis ${x\choose j}_q$ ($j\geq 0$). It turns out that 
the coefficients in such  expansions are Carlitz-Riordan's $q$-ballot numbers.  
Note that formula \eqref{eq:main2} 
was  implicitly given in \cite{Cig97}.

\begin{thm}\label{thm1} For $n\geq 0$ we have 
\begin{align}\label{eq:main1}
C_{n+1}(x|q)&=\sum_{j=0}^{n}f(n+j,n-j|q^{-1}) q^{jn+\frac{1}{2}(n-j)(n+j+1)}
{x\choose j}_q,\\
C_{n}(qx+1|q)&=\sum_{j=0}^{n-1}f(n+j,n-1-j|q^{-1})
q^{jn+\frac{1}{2}n(n+1)-\frac{1}{2}(j+1)(j+2)}{x\choose j}_q.\label{eq:main2}
\end{align}
\end{thm}
\begin{proof}
It is sufficient to prove the theorem for $x=[k]_q$ with $k=0, 1, \ldots, n$.
By Proposition~\ref{prop1}, the two equations \eqref{eq:main1} and  \eqref{eq:main2}   are  equivalent to 
\begin{align}
f(k+n, n|q^{-1})&=\sum_{j=0}^{k}f(n+j,n-j|q^{-1}) q^{jn-kn-j}
{k\brack j}_q,\\
f(k+n, n-1|q^{-1})&=\sum_{j=0}^{k}f(n+j,n-j-1|q^{-1}) q^{jn-kn-2j+k}
{k\brack j}_q,
\end{align}
for $n\geq k\geq j$. Replacing $q$ by $1/q$ and using ${k\brack j}_{q^{-1}}=q^{-j(k-j)}{k\brack j}_q$ we get  
\begin{align}
f(n+k, n|q)&=\sum_{j=0}^{k}f(n+j,n-j|q) q^{(n-j)(k-j)+j}{k\brack j}_q,\label{eq:key1}\\
f(k+n, n-1|q)&=\sum_{j=0}^{k}f(n+j,n-j-1|q) q^{(n-j-1)(k-j)+j}
{k\brack j}_q.\label{eq:key2}
\end{align}
We only prove \eqref{eq:key1}.
By definition, the left-hand side $f(n+k, n|q)$ is the enumerative polynomial of lattice paths from $(0,0)$ to $(n+k+1, n)$
with $(1,0)$ as the last step.
Each such path $\gamma$ must cross the line $y=-x+2n$. Suppose it crosses this line at 
the point $(n+j, n-j)$,  
$0\leq j\leq  k$. Then the path corresponds to a unique pair $(\gamma_1, \gamma_2)$, where
$\gamma_1$ is a path from $(0,0)$ to $(n+j, n-j)$ and $\gamma_2$ is a path from $(n+j,n-j)$ to $(n+k, n)$.
It is clear that  the area under the path $\gamma$ is equal to $S_1+S_2+S_3+j$, where
\begin{itemize}
\item $S_1$ is the area under the path $\gamma_1'$, which is obtained from $\gamma_1$ plus the last step 
$(n+j, n-j)\rightarrow (n+j+1,n-j)$;  
\item $S_2$ is the area under the path $\gamma_2$ and above the line $y=n-j$;
\item $S_3$ is the area of the rectangle delimited by the four  lines $y=0$, $y=n-j$, $x=n+j+1$ and $x=n+k+1$, i.e., $(n-j)(k-j)$.
\end{itemize}
This decomposition is depicted  in Figure~\ref{fig1}. Clearly, summing over all such lattice paths gives the summand on the right-hand side 
of \eqref{eq:key1}. This completes the proof.

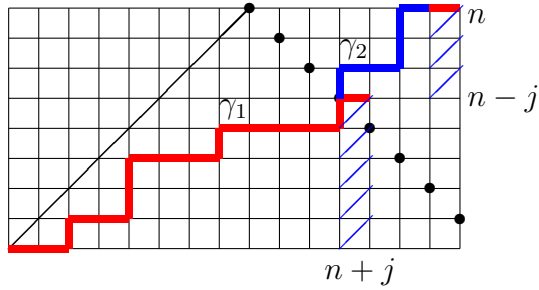
\begin{figure}[t]
\unitlength=10mm
\begin{picture}(6, 4)(0,0)
\multiput(.2,.2)(.4,0){16}
{\line(0,1){3.2}}
\multiput(.2,.2)(0, .4){9}
{\line(1,0){6}}
\put(.2,.2){\line(1,1){3.2}}
\put(3.4,3.4){\circle*{.15}}
\put(3.8,3){\circle*{.15}}
\put(4.2,2.6){\circle*{.15}}
\put(4.6,2.2){\circle*{.15}}
\put(5,1.8){\circle*{.15}}
\put(5.4,1.4){\circle*{.15}}
\put(5.8,1.0){\circle*{.15}}
\put(6.2,0.6){\circle*{.15}}
\put(.2,.2){\line(1,1){3.2}}
\put(6.3,3.2){$n$}
\put(6.3,2.1){$n-j$}
\put(4.4,-.2){$n+j$}
\linethickness{1mm}
\put(.2,.2){\red{\line(1,0){.8}}}
\put(1,.2){\red{\line(0,1){.4}}}
\put(1,.6){\red{\line(1,0){.8}}}
\put(1.8,.6){\red{\line(0,1){.8}}}
\put(1.8,1.4){\red{\line(1,0){1.2}}}
\put(3.0,2){$\gamma_1$}
\put(3.0,1.4){\red{\line(0,1){.4}}}
\put(3.0,1.8){\red{\line(1,0){1.6}}}
\put(4.6,1.8){\red{\line(0,1){.4}}}
\put(4.6,2.2){\red{\line(1,0){.4}}}
\put(4.6,2.2){\blue{\line(0,1){.4}}}
\put(4.6,2.6){\blue{\line(1,0){.8}}}
\put(5.4,2.6){\blue{\line(0,1){.8}}}
\put(5.4,3.4){\blue{\line(1,0){0.4}}}
\put(5.8,3.4){\red{\line(1,0){0.4}}}
\linethickness{2mm}
\put(4.6,.2){\blue{\line(1,1){.44}}}
\put(4.6,.6){\blue{\line(1,1){.44}}}
\put(4.6,1){\blue{\line(1,1){.44}}}
\put(4.6,2.8){$\gamma_2$}
\put(4.6,1.4){\blue{\line(1,1){.44}}}
\put(4.6,1.8){\blue{\line(1,1){.44}}}
\put(5.8,2.2){\blue{\line(1,1){.44}}}
\put(5.8,2.6){\blue{\line(1,1){.44}}}
\put(5.8,3){\blue{\line(1,1){.44}}}
\end{picture}
\caption{The decomposition $\gamma\mapsto (\gamma_1, \gamma_2)$}
\label{fig1}
\end{figure}
\end{proof}
\begin{rmk}
When $q=1$, by \eqref{eq:ballot},  the above theorem implies that
\begin{align}\label{eq:cq=1}
\frac{x+1}{x+n+1}{x+2n\choose n}&=\sum_{j=0}^{n}\frac{2j+1}{n+j+1}{2n\choose n-j}{x\choose j},\\
\frac{x+2}{x+n+1}{x+2n-1\choose n-1}&=\sum_{j=0}^{n-1}\frac{2j+2}{n+j+1}{2n-1\choose n-j-1}{x\choose j}.
\label{eq:dq=1}
\end{align}


Note that  the two  sequences 
$$\{f(n+j,n-j|1)\}\quad 
\text{and} \quad \{f(n+j,n-j-1|1)\}\qquad (0\leq j\leq n)
$$
 correspond, respectively, to  
the $(2n-1)$-th and $2n$-th anti-diagonal coefficients 
of the triangle $\{f(n,k)\}_{0\leq k\leq n}$, see Table~\ref{table1}.

%
\end{rmk}
\medskip
\begin{thm}\label{thm2} The polynomials $C_n(x|q)$ satisfy  
$C_1(x|q)=1$ and 
\begin{align}\label{eq:rec}
[n]_qC_{n+1}(x|q)=([2n-1]_q&+xq^{2n-1})C_n(x|q)\\
&+\sum^{n-2}_{j=0}[n-j-1]_q{\widetilde C}_j(q)C_{n-j}(x|q)q^{2j+1}. \nonumber
\end{align}
\end{thm}
\begin{proof} Since $C_{n+1}(x|q)$ is a polynomial in $x$ of degree $n$,  it suffices to
prove \eqref{eq:rec} for $x=[k]_q$, where $k$ is any positive integer, namely,
\begin{align}\label{x=k}
[n]_qC_{n+1}([k]_{q}|q)&=[2n+k-1]_qC_n([k]_{q}|q)\\
&+\sum^{n-2}_{j=0}[n-j-1]_q{\widetilde C}_j(q)C_{n-j}([k]_{q}|q)q^{2j+1}.\nonumber
\end{align}
 Let $m\geq n$ and
\begin{align}\label{eq:ftilde}
{\tilde f}(m,n|q)=q^{(m-n)n+{n+1\choose 2}}f(m,n|q^{-1}).
\end{align}
In view of the  definition \eqref{def:carlitz} it is clear that  
$$
{\tilde f}(m,n|q)=\sum_{\gamma\in {\mathcal P}(m,n)}q^{A'(\gamma)},
$$
where $A'(\gamma)$ denotes the area 
above the path $\gamma$ and under the line $y=x$ and $y=n$.
Since
$$
{\widetilde C}_j(q)=q^{j\choose 2}C_j(q^{-1})=q^{j+1\choose 2}f(j,j|q^{-1})=\tilde f(j,j|q),
$$
 using  Proposition~\ref{prop1} and \eqref{eq:ftilde} with $m=k+n$, we can rewrite  \eqref{x=k} as
\begin{align}\label{eq:key}
[n]_{q}&{\tilde f}(m, n|q)
=[n+m-1]_{q}{\tilde f}(m-1, n-1|q)\\
&+\sum^{n-2}_{j=0}q^{j}[n-j-1]_{q}{\tilde f}(j,j|q) q^{j+1}[n-j-1]_{q}{\tilde f}(m-j-1, n-j-1|q).\nonumber
\end{align}
A pointed lattice path is a pair $(\alpha, \gamma)$ such 
that $\alpha\in \{(1,1),\ldots, (n,n)\}$ and $\gamma\in {\mathcal P}(m,n)$.
If $\alpha=(i,i)$ we  call $i$ the height of $\alpha$ and write $h(\alpha)=i$.
Let ${\mathcal P}^*(m,n)$ be the set of all such pointed lattice paths.  
It is clear that the left-hand side of \eqref{eq:key} has the following interpretation
\begin{align}\label{eq:left}
[n]_{q}{\tilde f}(m, n|q)=\sum_{(\alpha,\gamma)\in {\mathcal P}^*(m,n)}q^{h(\alpha)-1+A'(\gamma)}.
\end{align}
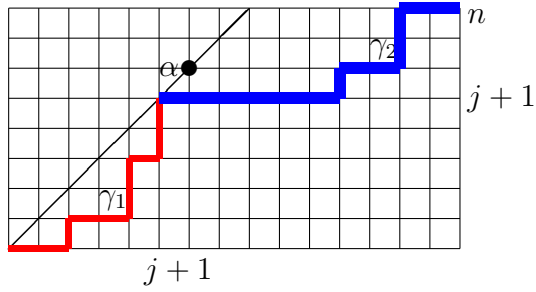
\begin{figure}[t]
\unitlength=10mm
\begin{picture}(6, 4)(0,0)
\multiput(.2,.2)(.4,0){16}
{\line(0,1){3.2}}
\multiput(.2,.2)(0, .4){9}
{\line(1,0){6}}
\put(.2,.2){\line(1,1){3.2}}
\put(.2,.2){\line(1,1){3.2}}
\put(6.3,3.2){$n$}
\put(2.2,2.5){$\alpha$}
\put(1.4, 0.8){$\gamma_{1}$}
\put(5, 2.8){$\gamma_{2}$}
\put(2.6,2.6){\circle*{.2}}
\put(6.3,2.1){$j+1$}
\put(2.0,-.2){$j+1$}
\linethickness{0.8mm}
\put(.2,.2){\red{\line(1,0){.8}}}
\put(1,.2){\red{\line(0,1){.4}}}
\put(1,.6){\red{\line(1,0){.8}}}
\put(1.8,.6){\red{\line(0,1){0.8}}}
\put(1.8,1.4){\red{\line(1,0){0.4}}}
\put(2.2,1.4){\red{\line(0,1){.8}}}
\linethickness{1.5mm}
\put(2.2,2.2){\blue{\line(1,0){2.4}}}
\put(4.6,2.2){\blue{\line(0,1){.4}}}
\put(4.6,2.6){\blue{\line(1,0){.8}}}
\put(5.4,2.6){\blue{\line(0,1){.8}}}
\put(5.4,3.4){\blue{\line(1,0){0.8}}}
\end{picture}
\caption{$(\alpha, \gamma)\to (\gamma_1, (\alpha', \gamma_2))$ }
\label{fig2}
\end{figure}
 Now, we compute the above enumerative polynomial of  ${\mathcal P}^*(m,n)$ in another way in order to 
obtain the right-hand side of \eqref{eq:key}.  We distinguish two cases.
\begin{itemize}
\item
Let ${\mathcal P}^*_1(m,n,j)$ 
be the set of all pointed lattice paths $(\alpha, \gamma)$ in ${\mathcal P}^*(m,n)$ such that 
  $h(\alpha)\in \{j+2,\ldots, n\}$,
  where $j$ is
   the smallest integer such that $(j+1,j)\rightarrow (j+1,j+1)$ is a step of $\gamma$, i.e., the 
    first step of $\gamma$ touching
   the line $y=x$. 
    If $(\alpha, \gamma)\in {\mathcal P}^*_1(m,n,j)$, then we have the correspondence 
$(\alpha, \gamma)\to (\gamma_1, (\alpha', \gamma_2))$,
where $\gamma_1$ is a lattice path from
$(0,0)$ to $(j+1, j+1)$ which touches the line $y=x$ only at the two extremities, and $(\alpha',\gamma_2)$ 
is a pointed lattice path from 
$(0,0)$ to 
$(m-j, n-j-1)$ with $h(\alpha')=h(\alpha)-j-1$.  This decomposition is depicted in Figure~\ref{fig2}.

Thus  the corresponding enumerative polynomial of such paths for the fixed $j$  is
 $$
\sum_{(\alpha,\gamma)\in {\mathcal P}^*_1(m,n,j)}q^{h(\alpha)-1+A(\gamma)}= q^{j}{\tilde f}(j,j|q)\cdot  q^{j+1} [n-j-1]_{q}{\tilde f}(m-j-1, n-j-1|q).
 $$
Summing over all $j$ ($0\leq j\leq n-2$) we obtain the second term on the right-hand side of \eqref{eq:key}.

\item Let ${\mathcal P}^*_2(m,n)$ 
be the set of all pointed lattice paths $(\alpha, \gamma)$ in ${\mathcal P}^*(m,n)$ such that 
 $h(\alpha)\in \{1, \ldots, n\}$ and  $h(\alpha)\leq j+1$ where   $j$ (if any) is
   the smallest integer such that $(j+1,j)\rightarrow (j+1,j+1)$ is a step of $\gamma$, i.e., the 
    first step of $\gamma$ touching
   the line $y=x$.
If  $(\alpha, \gamma)\in {\mathcal P}^*_2(m,n)$, 
 where
$\gamma=(p_{0},\ldots, p_{m+n+1})$ with $p_0=(0,0)$ and $p_{m+n+1}=(m+n+1, n)$,  we can associate 
a  pair  $(i, \gamma')$ where 
$\gamma'\in  {\mathcal P}(m-1,n-1)$ is obtained from 
 $\gamma$ by deleting the vertical step  $(x, h(\alpha)-1)\to (x, h(\alpha))$ and 
 the first horizontal step $(0,0)\to (1,0)$, i.e.,
 $$
 \gamma'=(p'_1, \ldots, p'_i, p'_{i+2}, \ldots, p'_{n+m+1})
 $$
  where
 $i=x+h(\alpha)-1$, $p'_k=p_k-(1,0)$ if $k=1, \ldots, i$ and $p'_k=p_k-(0,1)$ if $k=i+2, \ldots, m+n+1$.
 It is easy to see that the mapping $(\alpha, \gamma)\mapsto (i,\gamma')$ is a bijection, which 
 is depicted in Figure~\ref{fig3}.
\begin{figure}[t]
\unitlength=10mm
\begin{picture}(6, 4)(4,0)
\multiput(.2,.2)(.4,0){16}{\line(0,1){3.2}}
\multiput(.2,.2)(0, .4){9}{\line(1,0){6}}
\put(.2,.2){\line(1,1){3.2}}
\put(.2,.2){\line(1,1){3.2}}
\put(6.3,3.2){$n$}
\put(2.2,2.5){$\alpha$}
\put(5.6,-.2){$m$}
\linethickness{.5mm}
\put(.2,.2){\blue{\line(1,0){.4}}}
\put(.6,.2){\red{\line(1,0){.4}}}
\put(1,.2){\red{\line(0,1){.4}}}
\put(1,.6){\red{\line(1,0){.8}}}
\put(1.8,.6){\red{\line(0,1){.8}}}
\put(1.8,1.4){\red{\line(1,0){1.2}}}
\put(3.0,1.4){\red{\line(0,1){.4}}}
\put(3.0,1.8){\red{\line(1,0){1.6}}}
\put(4.6,1.8){\red{\line(0,1){.4}}}
\put(4.6,2.2){\blue{\line(0,1){.4}}}
\put(4.6,2.6){\red{\line(1,0){.8}}}
\put(5.4,2.6){\red{\line(0,1){.8}}}
\put(5.4,3.4){\red{\line(1,0){0.8}}}
\linethickness{.05mm}
\put(4.2,2.2){\blue{\line(1,1){.4}}}\put(4.2,2.6){\blue{\line(1,-1){.4}}}
\put(3.8,2.2){\blue{\line(1,1){.4}}}\put(3.8,2.6){\blue{\line(1,-1){.4}}}
\put(3.4,2.2){\blue{\line(1,1){.4}}}\put(3.4,2.6){\blue{\line(1,-1){.4}}}
\put(3,2.2){\blue{\line(1,1){.4}}}\put(3,2.6){\blue{\line(1,-1){.4}}}
\put(2.6,2.2){\blue{\line(1,1){.4}}}\put(2.6,2.6){\blue{\line(1,-1){.4}}}
\put(2.2,2.2){\blue{\line(1,-1){.4}}}\put(2.2,1.8){\blue{\line(1,1){.4}}}
\put(1.8,1.8){\blue{\line(1,-1){.4}}}\put(1.8,1.4){\blue{\line(1,1){.4}}}
\put(1.4,1.4){\blue{\line(1,-1){.4}}}\put(1.4,1.0){\blue{\line(1,1){.4}}}
\put(1,1){\blue{\line(1,-1){.4}}}\put(1,0.6){\blue{\line(1,1){.4}}}
\put(.6,.6){\blue{\line(1,-1){.4}}}\put(.6,.2){\blue{\line(1,1){.4}}}
\multiput(8.2,.2)(.4,0){16}{\line(0,1){3.2}}
\multiput(8.2,.2)(0, .4){9}{\line(1,0){6}}
\put(8.2,.2){\line(1,1){3.2}}
\put(6.6, 2){\vector(1,0){1}}
\put(14.3,3.2){$n$}
\put(12.2,2.2){\circle*{.2}}
\put(13.6,-.2){$m$}
\linethickness{.5mm}
\put(8.2,.2){\red{\line(1,0){.4}}}
\put(8.6,.2){\red{\line(0,1){.4}}}
\put(8.6,.6){\red{\line(1,0){.8}}}
\put(9.4,.6){\red{\line(0,1){.8}}}
\put(9.4,1.4){\red{\line(1,0){1.2}}}
\put(10.6,1.4){\red{\line(0,1){.4}}}
\put(10.6,1.8){\red{\line(1,0){1.6}}}
\put(12.2,1.8){\red{\line(0,1){.4}}}
\put(12.2,2.2){\red{\line(1,0){.8}}}
\put(13,2.2){\red{\line(0,1){.8}}}
\put(13,3){\red{\line(1,0){0.8}}}
\linethickness{1mm}
\put(4.2,2.2){\blue{\line(1,1){.4}}}\put(4.2,2.6){\blue{\line(1,-1){.4}}}
\put(3.8,2.2){\blue{\line(1,1){.4}}}\put(3.8,2.6){\blue{\line(1,-1){.4}}}
\put(3.4,2.2){\blue{\line(1,1){.4}}}\put(3.4,2.6){\blue{\line(1,-1){.4}}}
\put(3,2.2){\blue{\line(1,1){.4}}}\put(3,2.6){\blue{\line(1,-1){.4}}}
\put(2.6,2.2){\blue{\line(1,1){.4}}}\put(2.6,2.6){\blue{\line(1,-1){.4}}}
\put(2.2,2.2){\blue{\line(1,-1){.4}}}\put(2.2,1.8){\blue{\line(1,1){.4}}}
\put(1.8,1.8){\blue{\line(1,-1){.4}}}\put(1.8,1.4){\blue{\line(1,1){.4}}}
\put(1.4,1.4){\blue{\line(1,-1){.4}}}\put(1.4,1.0){\blue{\line(1,1){.4}}}
\put(1,1){\blue{\line(1,-1){.4}}}\put(1,0.6){\blue{\line(1,1){.4}}}
\put(.6,.6){\blue{\line(1,-1){.4}}}\put(.6,.2){\blue{\line(1,1){.4}}}
\end{picture}
\caption{$(\alpha, \gamma)\mapsto (i, \gamma')$ with $m=15$, $n=8$, $\alpha=(6,6)$ and $i=15$}
\label{fig3}
\end{figure}
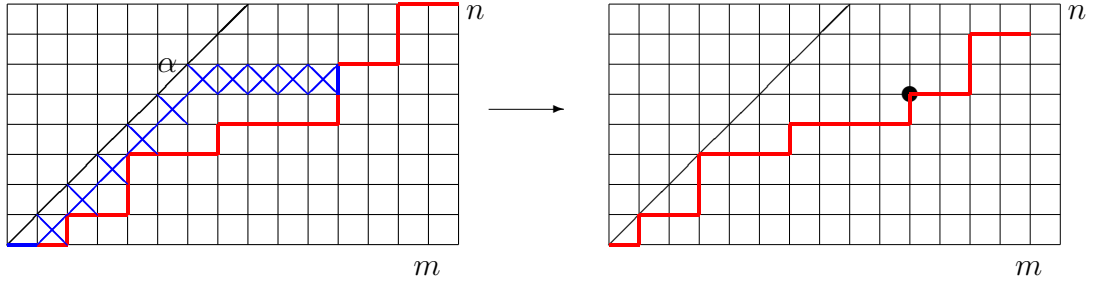
\end{itemize}
Since $1\leq x\leq m$ and $0\leq h(\alpha)-1\leq n-1$ we have  
 $i\in \{1, \ldots, m+n-1\}$. As $A'(\gamma)=x-1+A'(\gamma')$ we have 
 $$
 h(\alpha)-1+A'(\gamma)=i-1+A'(\gamma').
 $$ 
It follows that 
 $$
 \sum_{(\alpha, \gamma)\in {\mathcal P}^*_2(m,n)}q^{h(\alpha)+A'(\gamma)}
 =\sum_{i=1}^{m+n-1} 
 q^{i-1}\sum_{\gamma'} q^{A'(\gamma')}=[n+m-1]_{q}{\tilde f}(m-1, n-1|q).
 $$
Summing up the two cases we obtain the right-hand side of \eqref{eq:key}.
\end{proof}
When $q=1$ we have  an alternative proof of Theorem~\ref{thm2}.
\begin{proof}[Another proof of the $q=1$ case]
When $q=1$ Eq. \eqref{eq:rec} reduces to
\begin{align}\label{q=1}
nC_{n+1}(x|1)=(2n-1+x)C_n(x|1)+\sum^{n-2}_{j=0}(n-j-1)C_jC_{n-j}(x|1)\quad (n\geq 2).
\end{align}
This yields  immediately 
$C_1(x|1)=1$, $C_2(x|1)=x+1$, $C_3(x|1)=(x+1)(x+4)/2$, in accordance with
 the formula \eqref{eq:explicit}. For $n\geq 3$, letting $k=n-j-3$, $N=n-3$ and $z=x+3$,
 by \eqref{eq:explicit}, the 
 recurrence \eqref{q=1} is equivalent to the following identity
$$
\frac{(z+N+2)_{N}}{N!}=\sum_{k=0}^{N} 4^{N-k}\frac{(3/2)_{N-k}(z+k)_{k}}{(3)_{N-k}k!}\qquad (N\geq 0).
$$
Notice that we can rewrite the right-hand side as
\begin{align*}
\frac{(3/2)_{N}}{(3)_{N}}4^{N}&\sum_{k=0}^{N}\frac{(-2-N)_{k}((z+1)/2)_{k}(z/2)_{k}}{k!(-1/2-N)_{k}(z)_{k}}\\
=\frac{(3/2)_{N}}{(3)_{N}}4^{N}&\left({}_{3}F_{2}
\left(\begin{array}{ccc}
-2-N,&(z+1)/2,&z/2\\
-1/2-N,&z,&
\end{array};1\right)\right.\\
&\quad -\frac{(-2-N)_{N+1}((z+1)/2)_{N+1}(z/2)_{N+1}}{(-1/2-N)_{N+1}(z)_{N+1}(N+1)!}\\
&\quad\quad\left.-\frac{(-2-N)_{N+2}((z+1)/2)_{N+2}(z/2)_{N+2}}{(-1/2-N)_{N+2}(z)_{N+2}(N+2)!} \right).
\end{align*}
Invoking  Pfaff-Saalsch\"utz formula \cite[Theorem 2.2.6]{AAR99} we obtain 
$$
{}_{3}F_{2}
\left(\begin{array}{ccc}
-2-N,&(z+1)/2,&z/2\\
-1/2-N,&z,&
\end{array};1\right)=\frac{(z/2)_{N+2}((z-1)/2)_{N+2}}{(z)_{N+2}(-1/2)_{N+2}}.
$$
Substituting this in the previous expression  yields  $\frac{(z+N+2)_{N}}{N!}$ after simplification.
\end{proof}

When $x=1$ Eq.~\eqref{q=1} reduced to the following identity for Catalan numbers:
 \begin{align}
 nC_{n+1}=2nC_n+\sum_{j=0}^{n-2}(n-j-1)C_jC_{n-j}.
 \end{align}

\section{Concluding remarks}
\bigskip

We conclude this paper with a few open problems. By Theorem~\ref{thm1}
it is clear that $C_{n+1}(x|q)$ is a polynomial in $x$  of degree $n$ with leading coefficient
$q^{n^2}/[1]_q[2]_q\cdots [n]_q$. 

\begin{conj}
  One can write $C_{n}(x|q)$ as an irreducible fraction
  \begin{equation*}
    \frac{P_n(x|q)}{[1]_q\cdots [n-1]_q},
  \end{equation*}
  where $P_n$ has only positive coefficients.
\end{conj}
This has been checked up to $n=27$. The similar conjecture is true
when $q=1$ by Proposition \ref{prop2}.

%

Finally, the Newton polytope of the numerator of $C_{n}(x|q)$ seems to
have a nice shape. This is illustrated in Figure~\ref{newton}, where
the horizontal axis is associated with powers of $q$ and the vertical
axis with powers of $x$. The slopes of the upper part seems to be
given in general by the odd integers $1$, $3$, \dots, $2n-3$.

\begin{figure}\centering
  \includegraphics[height=3cm]{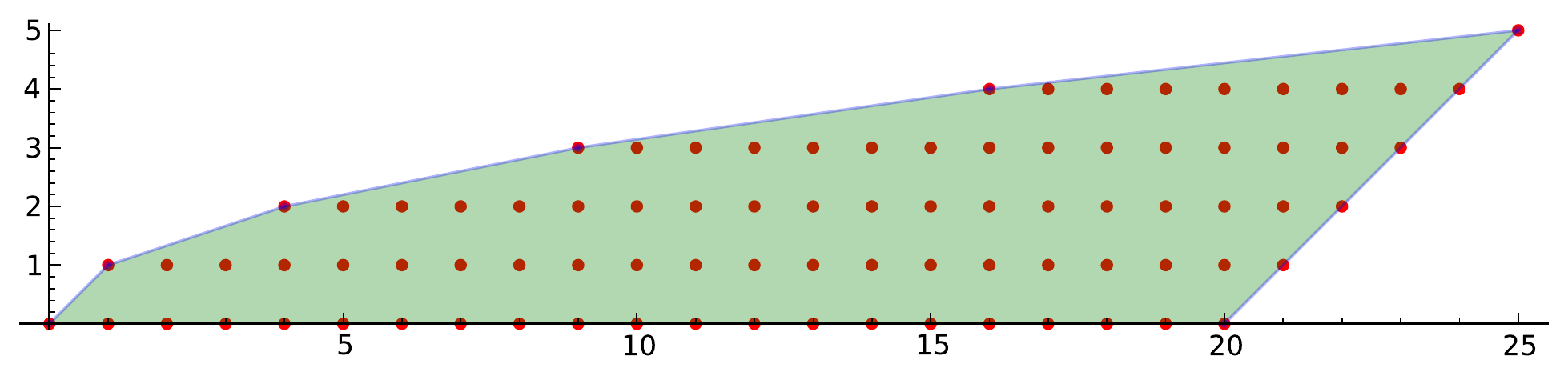}
  \caption{Newton polytope of the numerator of $C_{n}(x|q)$ for $n=6$}
  \label{newton}
\end{figure}

%

\end{document}